\documentclass{elsart}
\date{}
\usepackage{graphicx}

\journal{RAIRO }

\usepackage{amssymb, amsmath,amssymb}
\usepackage{graphicx}
\usepackage[T1]{fontenc}
\usepackage{enumerate}
\usepackage{graphicx}
\usepackage{float}
\usepackage{curves}
\usepackage{rotfloat}
\newtheorem{theorem}{Theorem}[section]

\newtheorem{lemma}[theorem]{Lemma}
\newtheorem{proposition}[theorem]{Proposition}
\newenvironment{proof}[1][Proof]{\noindent\textbf{#1. }}{\ \rule{0.5em}{0.5em}}

\theoremstyle{definition}
\newtheorem{definition}[theorem]{Definition}
\newtheorem{remark}[theorem]{Remark}

\begin{document}

\begin{frontmatter}

\title{New concept of connection in signed graphs}

\author[add1]{Ouahiba Bessouf}
\ead{obessouf@yahoo.fr}
\author[add1]{Abdelkader Khelladi}
\address[add1]{Faculté de Mathématiques, USTHB BP 32 El Alia, Bab-Ezzouar 16111, Alger, Algérie}
\ead{kader{\_}khelladi@yahoo.fr}
\author[add2]{Thomas Zaslavsky}
\address[add2]{Department of Mathematical Sciences, Binghamton University, Binghamton, NY 13902-6000, U.S.A.}
\ead{zaslav@math.binghamton.edu}

\begin{abstract}
In a signed graph each edge has a sign, $+1$ or $-1$. We introduce in the present paper a new definition of connection in a signed graph by the existence of both positive and negative chains between vertices.  We prove some results and properties of this definition, such as sign components, sign articulation vertices, and sign isthmi, and we compare them to corresponding graph and signed-graphic matroid properties.  
We apply our results to signed graphs without positive cycles.  For signed graphs in which every edge is negative our properties become parity properties.
\end{abstract}

\thanks{Version of 27 January 2017}

\end{frontmatter}


	\section{Introduction}

A signed graph is a graph in which every edge has a sign, $+1$ or $-1$.  Signed graphs, as a generalization of ordinary undirected graphs, were introduced by Harary in 1954.  The aim of this paper is to introduce a new concept of connection of a signed graph, which we call sign connection; it means that every pair of vertices is joined by both a positive and a negative chain.  

We define and prove properties of sign connection such as elementary signed chains, sign isthmus, sign articulation vertex, sign block, etc., and we establish relationships to connection in the graph and in matroids of the signed graph.  We apply our results to the class of signed graphs without positive cycles.  We also explain how they apply to the class of signed graphs where all signs are negative; then we are asking about graphs in which every vertex pair be joined by both odd and even chains.  We do not know of previous work on this concept.  We conclude by discussing alternative notions of connection in a signed graph, in particular connection by chains of only one sign.

A main theorem is that sign connection, though a very natural concept, is equivalent to either graph connection or total disconnection, depending on the graph.  However, the properties we find of sign isthmi, which are more like matroid coloops than graph isthmi, and sign articulation vertices suggest that higher sign connectivity may be a new kind of connectivity, intermediate between graph and matroid connectivity, that will repay further examination.

	\section{Signed graphs}\label{signed}

We consider finite, undirected graphs $G$ with the vertex set $V=V(G)$ and the edge set $E=E(G)$.  A graph may have loops and multiple edges.  
The chains which are used in this paper are the usual chains (or walks) of undirected graphs as in \cite{bb}.  
All cycles are elementary, that is, without self-intersections.
An \emph{articulation vertex} in $G$ is a vertex $v$ such that there exist edges, $e$ and $f$, for which every chain from $e$ to $f$ passes through $v$.  For instance, a vertex that supports a loop is an articulation vertex unless it is incident with no other edge.  A \emph{block} of $G$ is a maximal subgraph with no articulation vertex; for instance, the subgraph induced by a loop is a block, and an isolated vertex is a block.  An \emph{isthmus} of $G$ is an edge $e$ such that $G-e$ has more connected components than $G$.

\begin{definition}\label{1}
{\rm
A \emph{signed graph} is a triple $(V, E; \sigma  )$ where $G = (V, E)$
is an undirected graph and $\sigma$ is  a signature of the edge set $E$:
\begin{align*}
\sigma : E  &\rightarrow \{ -1, +1\}\\
         e  &\mapsto \sigma(e)
\end{align*}
A signed graph is denoted by $G_{\sigma}= (V, E;\sigma )$.  Sometimes we write $+,-$ for signs instead of $+1,-1$.
}
\end{definition}

\begin{definition}\label{2}
{\rm
A cycle of a signed graph is \emph{positive} if the sign product of its edges is positive, or the number of its negative edges is even.  In the opposite case it is \emph{negative}.  
A signed graph is \emph{balanced} if all its cycles are positive; e.g., if it is cycle free (i.e., a forest).
}
\end{definition}

\begin{definition}\label{switch}
{\rm
\emph{Switching} a signed graph means reversing the signs of all edges between a vertex subset $W \subseteq V$ and its complement, $V-W$.  (The negated edge set may be void.)  We note that switching does not change the sign of a closed chain.
}
\end{definition}

Harary's balance theorems \cite{h} may be restated as follows:  

\begin{lemma}\label{balance}
A signed graph $G_{\sigma}$ is balanced if and only if it can be switched to have all positive signs.  Also, $G_{\sigma}$ is balanced if and only if for every $x,y \in V$ (possibly equal), every chain between $x$ and $y$ has the same sign.
\end{lemma}

\begin{definition}\label{harbi}
{\rm
Let $W$ be the set of vertices that are switched in Lemma \ref{balance}.  The pair $\{ W, V-W \}$ is called a \emph{Harary bipartition} of $G_\sigma$.  It is unique if and only if $G_\sigma$ is connected.
}
\end{definition}

\begin{definition}\label{baledge}
Let $G_\sigma$ be a connected, unbalanced signed graph.  An edge $e$ such that $G_\sigma-e$ is balanced, is called a \emph{balancing edge} of $G_\sigma$.
\end{definition}

\begin{proposition}\label{p:baledge}
Let $G_{\sigma}$ be a connected, unbalanced signed graph and $e$ an edge.  The following properties of $e$ are equivalent:
\begin{enumerate}[{\rm (1)}]
\item $e$ is a balancing edge.
  \label{p:baledge-baledge}
\item $e$ belongs to every negative cycle.
  \label{p:baledge-negcyc}
\item $e$ belongs to every negative cycle and does not belong to any positive cycle.%
  \label{p:baledge-poscyc}
\item $e$ is not an isthmus, $G_\sigma - e$ is balanced, and $\sigma(e)$ differs from the sign of a chain connecting its endpoints in $G_\sigma - e$.
  \label{p:baledge-chain}
\item $G_\sigma$ switches so that $E-e$ is all positive and $e$ is negative.
  \label{p:baledge-posneg}
\end{enumerate}
\end{proposition}
\begin{proof}
\eqref{p:baledge-chain} $\Leftrightarrow$ \eqref{p:baledge-posneg} by Lemma \ref{switch}.

\eqref{p:baledge-baledge} $\Leftrightarrow$ \eqref{p:baledge-negcyc} follows from the definition.

\eqref{p:baledge-baledge} $\Rightarrow$ \eqref{p:baledge-posneg}.  If $e$ is a balancing edge, then $G_\sigma$ switches so that $E-e$ is all positive (Lemma \ref{switch}).  If $e$ were positive, $G_\sigma$ would be balanced, which it is not; thus $e$ is negative.

\eqref{p:baledge-posneg} $\Rightarrow$ \eqref{p:baledge-poscyc} is evident.
\eqref{p:baledge-poscyc} $\Rightarrow$ \eqref{p:baledge-negcyc} is trivial.  
\end{proof}

\begin{definition}
{\rm 
In a signed graph $G_\sigma$, a block is an \emph{inner block} if it is unbalanced or it contains an edge of a path connecting two balanced blocks.  In the opposite case it is an \emph{outer block}.

An \emph{unbalanced necklace of balanced blocks} is a signed graph constructed from balanced signed blocks $B_{1\sigma_1}, B_{2\sigma_2}, \ldots, B_{k\sigma_k}$ ($k\geq2$) and distinct vertices $v_i, w_i \in B_{i\sigma_i}$ by identifying $v_i$ with $w_{i-1}$ for $i=2,\ldots,k$ and $v_1$ with $w_k$.  The blocks $B_{i\sigma_i}$ are called the \emph{constituents} of the necklace.  Note that an unbalanced necklace of balanced blocks is an unbalanced block.

A \emph{core} of $G_\sigma$ is the union of inner blocks of an unbalanced connected component of $G_\sigma$.
}
\end{definition}

\begin{definition}\label{4}
{\rm
Let $G_{\sigma}$ be a signed graph and let $P$ be a chain (not necessarily elementary) connecting $x$ and $y$ in $G_{\sigma}$:
$$P{:}\ x,e_{1},x_{1},e_{2},x_{2},\ldots,y.$$
where $x,x_{1},\ldots,y \in V$ and $e_{1},e_{2},\ldots \in E$.
We put 
$$\sigma(P)=\displaystyle{\prod_{e_{i}\in P}\sigma(e_{i})}\in \{-1,+1\}$$ 
and we call $P$ \emph{positive} (resp., \emph{negative}) if $\sigma(P)=+1$ (resp., $\sigma(P)=-1$).  
(This sign rule generalizes the sign of a cycle.)

We may write $P^{\varepsilon}$ instead of $P$ when $\varepsilon=\sigma(P)$ and we may write $P^{\varepsilon}(x, y)$ when $P$ is a chain connecting $x$ and $y$.
$P^{\varepsilon}$ is called an \emph{$\varepsilon$-chain of sign $\varepsilon$}.

The graph of a chain $P$ is $G(P)$, the subgraph of $G_\sigma$ that consists of all the vertices and edges of $P$.

An $\varepsilon$-chain $P^{\varepsilon}(x,y)$ is \emph{elementary} if it is minimal given its sign and its end vertices.  
Minimality means that no other $\varepsilon$-chain that connects $x$ and $y$ has a graph that is a proper subgraph of $G(P^{\varepsilon}(x,y))$.
}
\end{definition}

\begin{figure}[htbp]
  \begin{center}
  \includegraphics[scale=0.8]{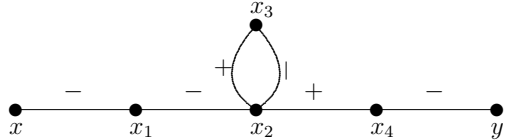}
  \end{center}
    \caption{$P^{+}(x, y){:}\ x,x_{1},x_{2},x_{3},x_{2},x_{4},y$ is a positive chain and $P^{-}(x, y){:}\ x,x_{1},x_{2},x_{4},y$ is a negative chain.  $P^{+}$ contains $P^{-}$ as a subchain, but both $P^{+}$ and $P^{-}$ are elementary $\varepsilon$-chains for different values of $\varepsilon$.}
    \label{FIG2}
\end{figure}

Now we give the different types of elementary $\varepsilon$-chain which admit a negative cycle.

\begin{definition}\label{6}
{\rm
A \emph{hypercyclic chain} $P^{\varepsilon}$ connecting two vertices (not necessarily
distinct) in a signed graph $G_{\sigma}$ is an elementary $\varepsilon$-chain which contains a negative cycle.
Note that the value of $\varepsilon$ does not affect the definition of a hypercyclic chain, but it must be specified.
}
\end{definition}

\begin{proposition}\label{hypercyclic}
Figure \ref{FIG3} shows the three possible cases of a hypercyclic chain.  
\end{proposition}
\begin{proof}
\textit{Correctness of the three types}:  
In each of the three types the whole figure is the graph of an elementary $\varepsilon$-chain connecting $x$ and $y$ that contains the negative cycle.  There is also an elementary $(-\varepsilon)$-chain that does not contain the negative cycle.  Thus, the whole figure is a hypercyclic $\varepsilon$-chain.  

\textit{Completeness of the three types}:  
Let $P^\varepsilon(x,y)$ be a hypercyclic chain in $G_\sigma$.  $S$ cannot be a tree because it must contain a negative cycle.  

If $G(P^\varepsilon)$ contains a positive cycle $C$ with an edge $e$, then $e$ can be replaced by $C-e$ anywhere it appears in $P^\varepsilon$ to get a new $\varepsilon$-chain $Q^\varepsilon$ connecting $x$ and $y$ whose graph is a proper subgraph of $G(P^\varepsilon)$.  $Q^\varepsilon$ can be simplified to an elementary $\varepsilon$-chain whose graph is contained in $G(Q^\varepsilon)$.  That contradicts minimality of $P^\varepsilon$.  Therefore, $G(P^\varepsilon)$ cannot contain a positive cycle.  It cannot contain two negative cycles that have more than one common vertex because their union contains a positive cycle.

Suppose $G(P^\varepsilon)$ contains two negative cycles, $C$ and $\acute C$.  If $C$ and $\acute C$ have no common vertex, then $G(P^\varepsilon)$ also contains an elementary chain $Q$ that connects them and is internally disjoint from both cycles.  If $C$ and $\acute C$ have a single common vertex, let $Q$ be the chain of length 0 that consists of that common vertex.
Let $F=C \cup \acute C \cup Q$ and let $e$ be an edge of $C$ with end vertices $u_1, u_2$.  ($F$ is a frame circuit of type (ii) or (iii); see Definition \ref{3}.)  
Then $e$ can be replaced in $P^\varepsilon$ by a chain in $F-e$ that connects $u_1$ and $u_2$; as with a positive cycle, this contradicts the minimality of $P^\varepsilon$.  

We conclude that $P^\varepsilon$ contains a unique cycle $C$, which is negative.  There are an elementary chain $A$ in $G(P^\varepsilon)$ that connects $x$ to a vertex $t \in C$ and an elementary chain $B$ that connects $y$ to a vertex $u\in C$.  

If $t\neq u$, then there are two elementary chains in $C$ from $t$ to $u$; we call them $C_1$ and $C_2$.  The elementary chains $P_1=AC_1B^{-1}$ and $P_2=AC_2B^{-1}$ have opposite signs.  A hypercyclic chain $P^\varepsilon(x,y)$ must have the form $AC_1C_2^{-1}C_1B^{-1}$ in order to contain the negative cycle, but then it is not minimal because either $P_1$ or $P_2$ is smaller and has the same sign $\varepsilon$.  Therefore $t=u$.

If $A$ and $B$ are internally disjoint, by minimality we have type (a) or (b) in Figure \ref{FIG3}.  If they are not internally disjoint, by minimality we have type (c).  That completes the proof.
\end{proof}

\begin{figure}[htbp]
   \includegraphics[scale=0.8]{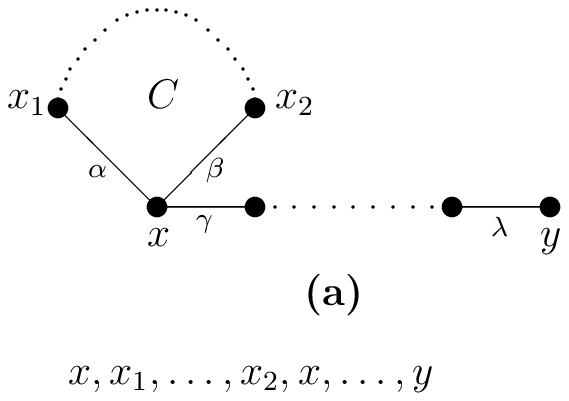}
   \hspace*{1cm}
   \includegraphics[scale=0.8]{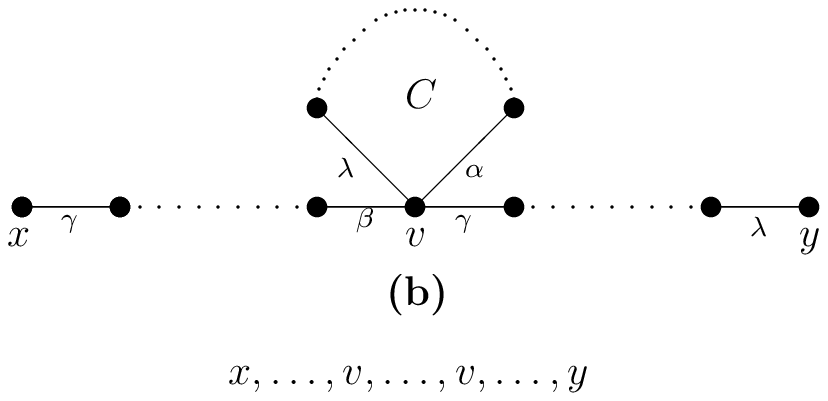}\\
   \begin{center}
  \includegraphics[scale=0.8]{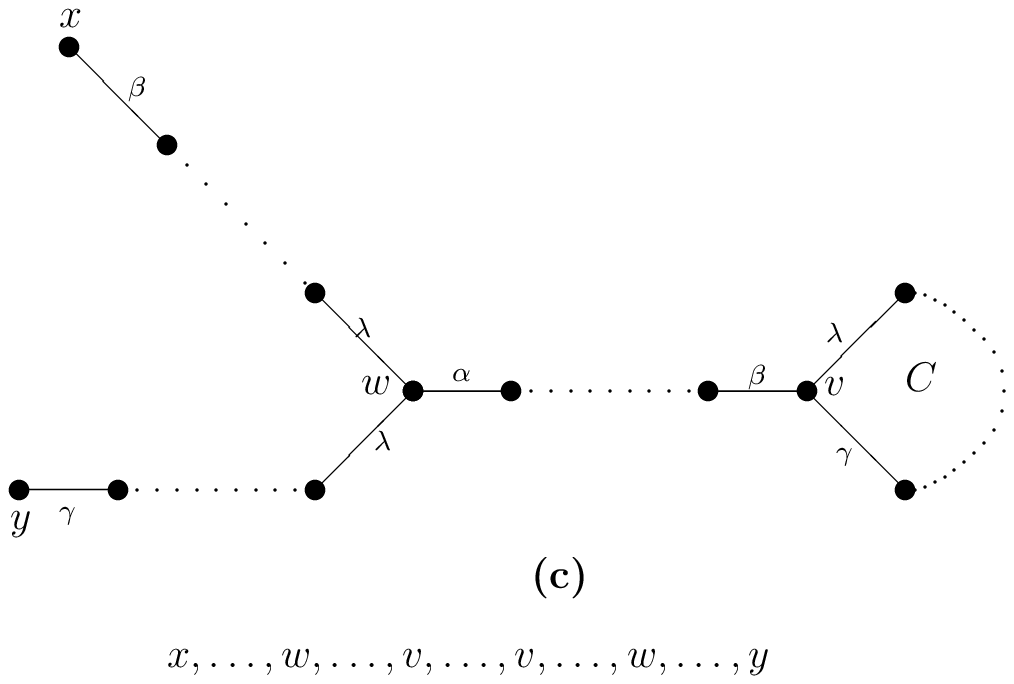}
  \caption{The three types of hypercyclic chain. We note that $\alpha, \beta, \gamma, \lambda  \in\{-1, +1\}$.  In (a) and (c), $y$ may coincide with $x$.  In (c), $x$ or $y$ (or both) may coincide with $w$.  Note that if in (c) we let $w=v$, we get (a) or (b).}
  \label{FIG3}
  \end{center}
\end{figure}

	\section{Sign connection}\label{signconn}

\begin{definition}\label{27}
{\rm
Let $G_{\sigma}= (V, E ;\sigma)$ be a signed graph. Let us consider the relation $R$ in $V$ defined by:
\begin{align*}
&x R y \Leftrightarrow \\
      &\begin{cases}
        x = y , \text{ or}\\
      \text{there exist both a negative chain and a positive chain joining $x$ and $y$}.
       \end{cases}
\end{align*}
Then $R$ is an equivalence relation whose equivalence classes form a partition $\{V_{1},\ldots, V_{q}\}$ of $V$. If $xRy$, we say $x$ and $y$ are \emph{sign connected}.  The subgraphs of $G_{\sigma}$ generated by the subsets $V_{i}$ ($i= 1, \ldots, q$) are  called the \emph{sign-connected components} of $G_{\sigma}$.

A signed graph is called \emph{sign connected} if it has just one sign-connected component; that is, for any pair of vertices $x$ and $y$, there exist both a positive chain and a negative chain connecting them.
}
\end{definition}

	\subsection{Properties of sign connection}\label{s-conn}

\begin{lemma}\label{29}
A hypercyclic chain is a sign-connected signed graph.
\end{lemma}
\begin{proof}
Let $P^{\alpha}(x, y)$ be a hypercyclic chain of  sign $\alpha$  connecting $x$ and $y$ (see figure \ref{FIG3}):
$$P^{\alpha}(x, y){:}\ x, x_{1}, x_{2}, \ldots, x_{i}, z_1, \ldots, z_j, x_{i}, x_{i+1}, \ldots, y.$$
The sequence $C$ from $P^{\alpha}(x, y)$, given by
$C{:}\ x_{i}, z_1, \ldots, z_j, x_{i}$, is a negative cycle. Let $P^{\beta}(x, y)$ be the following acyclic $\beta$-chain contained in $P^{\alpha}(x, y)$:
$$P^{\beta}(x, y){:}\ x, x_{1}, x_{2}, \ldots, x_{i}, x_{i+1}, \ldots, y.$$
Since $\beta$ is given by
$\alpha = \sigma(P^{\alpha}) = \sigma(P^{\beta}) \sigma(C) = - \beta,$ the two $\varepsilon$-chains $P^{\alpha}(x, y)$   and  $P^{\beta}(x, y)$ are of different signs.
\end{proof}

\begin{theorem}\label{31}
The sign components of a signed graph are the unbalanced connected components and the individual vertices of the balanced connected components.

A signed graph $G_{\sigma}$ is sign connected if and only if either $G_{\sigma}$ is connected and unbalanced, or $|V|=1$.
\end{theorem}

\begin{proof}
Assume $G_\sigma$ is sign connected.  Obviously, it must be connected.

\emph{Necessary condition}:  Assume that $G_{\sigma}$ is balanced.  
Assume that $G_{\sigma}$ has two distinct vertices, $x, y$, that are sign connected.  Thus, there are chains $P^+(x,y)$ and $P^-(y,x)$, whose concatenation is a negative closed chain.  After switching so $G_{\sigma}$ is entirely positive (by Lemma \ref{balance}), this closed chain is still negative, which is impossible.  Therefore, no such distinct $x,y$ can exist, so $|V|=1$.

\emph{Sufficient condition}:  Assume that $G_{\sigma}$ is unbalanced and let $C$ be a negative cycle.  Let $x,y$ be two vertices and let $P_x, P_y$ be paths (possibly of length 0) from $x, y$ (resp.)\ to $C$, chosen so that, if they have a common vertex, they are identical from their first common vertex on to $C$.  
If $P_x$ and $P_y$ have a common vertex, then $P_x \cup P_y \cup C$ is a hypercyclic chain, thus by Lemma \ref{29} $x$ and $y$ are sign connected.  
Otherwise, there are two paths in $C$ that connect the vertices $\acute{x}$ and $\acute{y}$ at which $P_x$ and $P_y$ (resp.)\ meet $C$.  These paths have opposite signs, so $x$ and $y$ are joined by two paths of opposite sign.
\end{proof}

	\subsection{Quasibalance}\label{qbal}

\begin{definition}\label{qbal-def} 
{\rm
A signed graph $G_{\sigma}$ is called \emph{quasibalanced} if for any two negative cycles $C$ and $\acute{C}$ of $G_{\sigma}$ we have
$$|V(C) \cap  V(\acute{C})|\geq 2.$$
}
\end{definition}

Quasibalance seems to be a new idea.  It is intimately related to sign connection and to matroids (see Section \ref{matroids}).  
We do not yet know how to completely characterize quasibalanced signed graphs, but Proposition \ref{qbal-disconn} follows easily from the definition. 

\begin{proposition}\label{qbal-disconn}
A signed graph is quasibalanced if and only if it has at most one unbalanced block and that block is quasibalanced.
\end{proposition}

	\subsection{Sign isthmus and sign articulation vertex}\label{s-isthmus}

\begin{definition}\label{32}
{\rm
Let $G_{\sigma} $ be a signed graph that is sign connected. An edge $e \in E$ is called a \emph{sign isthmus} if the graph $G_{\sigma} - e$ is not sign connected.
}
\end{definition}

\begin{proposition}\label{36}
Let $G_{\sigma}$ be a sign-connected signed graph with $|V|>1$.  An edge $e$ is a sign isthmus if and only if it is an isthmus or a balancing edge.
\end{proposition}
\begin{proof}
Let $e$ be a sign isthmus.  We distinguish the two following cases:

\emph{Case 1.} If $G_{\sigma} - e$ is not connected, then $e$ is an isthmus.  For the converse, an isthmus in a sign-connected signed graph is clearly a sign isthmus.

\emph{Case 2.} If $G_{\sigma} - e$ is connected, assume it is unbalanced.  Then it is sign connected (Theorem \ref{31}).  Thus, $e$ is not a sign isthmus.  Therefore $G_{\sigma} - e$ is balanced.  Since $G_{\sigma}$ is unbalanced (Theorem \ref{31}), $e$ is a balancing edge.  For the converse, if $e$ is a balancing edge then $G_{\sigma} - e$ is balanced so it is not sign connected.
\end{proof}

\begin{definition}\label{37}
{\rm
Let $G_{\sigma}$ be a sign-connected signed graph, and let $x \in V$. The vertex $x$ is called a \emph{sign articulation vertex} if the graph $G_{\sigma} - x$ is not sign connected.
}
\end{definition}

\begin{proposition}\label{38}
Let $G_{\sigma}$ be a sign-connected signed graph. If $x$ is a sign articulation vertex, then one and only one of the following properties holds:
\begin{enumerate}[{\rm (1)}]
\item $x$ is an  articulation vertex.
\item $x$ belongs to a negative cycle.
\end{enumerate}
\end{proposition}
\begin{proof}
Let $x$ be a sign articulation vertex.  We distinguish the two following cases:

\emph{Case 1.} If $G_{\sigma}- x$ is not connected, then $x$ is an  articulation vertex.

\emph{Case 2.} If $G_{\sigma}- x$ is connected, then $x$ belongs to a cycle. Since $G_{\sigma}- x$ is not sign connected, by Theorem \ref{31} it is balanced, and since $G_{\sigma}$ is unbalanced, then some cycle which admits the vertex $x$ is negative.
\end{proof}

\begin{proposition}\label{40}
Let $G_{\sigma}$ be a sign-connected signed graph. If $G_{\sigma}$ contains two vertex-disjoint negative cycles, then every sign articulation vertex is an articulation vertex.
\end{proposition}
\begin{proof}
Let $C_{1}$  and $C_{2}$  be two disjoint negative cycles of $G_{\sigma}$.  No vertex belongs to both of them, so $G_{\sigma}- x$ is unbalanced.  By Theorem \ref{31}, $G_{\sigma}- x$ is sign connected if it is connected.  Therefore, if $x$ is a sign articulation vertex, it is an articulation vertex.
\end{proof}

\begin{remark}\label{41}
{\rm 
Let $G_{\sigma}$ be a sign-connected signed graph and $x$ an articulation vertex of $G_{\sigma}$. If two or more of the connected components of $G_{\sigma}- x$ are sign connected, then each admits a negative cycle, so $G_{\sigma}$ is not quasibalanced.
}
\end{remark}

	\subsection{Sign block}\label{block}

\begin{definition}\label{42}
{\rm
Let $G_{\sigma}$ be a sign-connected signed graph. $G_{\sigma}$ is called a \emph{sign block} if $G_{\sigma}$ has no sign articulation vertex.
}
\end{definition}

\begin{proposition}\label{43}
Let $G_{\sigma}$ be a sign-connected signed graph. If $G_{\sigma}$ is a block that admits two vertex-disjoint negative circles, then it is a sign block.
\end{proposition}
\begin{proof}
It is enough to use Definition \ref{42} and Proposition \ref{40}.
\end{proof}

\begin{theorem}\label{45}
Let $G_{\sigma}$ be a sign-connected block with at least three vertices.  $G_{\sigma}$ is a sign block if and only if it has no balancing vertex.
\end{theorem}
\begin{proof}
Since $G_{\sigma}$ is sign-connected with at least two vertices, it is unbalanced.  
A vertex $x$ is a sign articulation vertex $\Leftrightarrow$ $G_{\sigma}- x$  is not sign connected $\Leftrightarrow$ (since $G_{\sigma}- x$ is connected) $G_{\sigma}- x$  is balanced and has at least two vertices (which it does, by assumption).  Therefore, $G_{\sigma}$ has no sign articulation vertex $\Leftrightarrow$ $G_{\sigma}- x$ is unbalanced for every vertex $x$.
\end{proof}

	\section{Sign connection and matroids}\label{matroids}

	\subsection{Frame and lift matroids}

\begin{definition}[\cite{z,z2,z3}]\label{3}
{\rm
A signed graph $G_{\sigma}=(V,E; \sigma)$ has associated two matroids, the \emph{frame matroid} $M(G_{\sigma})$ and the \emph{lift matroid} $L(G_\sigma)$. A definition of $M(G_\sigma)$ is that a subset $F \subseteq E$ is a circuit of $M(G_{\sigma} )$ (a \emph{frame circuit} of $G_\sigma$) if either
\begin{enumerate}[{\rm Type (i):}]
\item $F$ is a positive cycle,
\item or $F$ is the union of two negative cycles having exactly one common vertex,
\item or $F$ is the union of two vertex-disjoint negative cycles and an elementary chain which connects the cycles and is internally disjoint from both cycles.
\end{enumerate}
}
\begin{figure}[htbp]
  \begin{center}
  \includegraphics[scale=1]{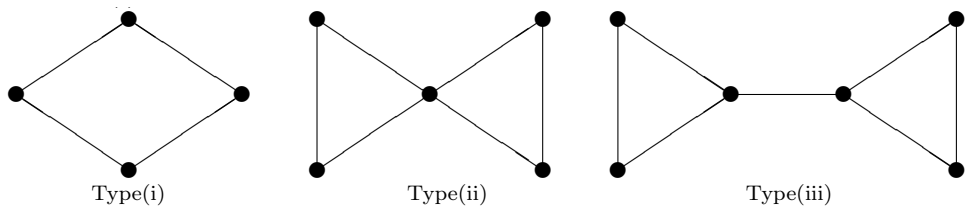}
  \end{center}
   \caption{We represent a positive (resp., negative) cycle by a quadrilateral (resp., triangle).}
   \label{FIG1}
\end{figure}

{\rm 
A definition of $L(G_\sigma)$ is that a subset $F \subseteq E$ is a circuit of $L(G_{\sigma} )$ (a \emph{lift circuit} of $G_\sigma$) if it is of type (i), (ii), or 
\begin{enumerate}[{\rm Type (i):}]
\item[{\rm Type (iii$'$):}] $F$ is the union of two vertex-disjoint negative cycles.
\end{enumerate}
}
\end{definition}

\begin{proposition}\label{9}
A signed graph does not admit circuits of types (ii) and (iii) if and only if every connected component is quasibalanced.

A signed graph does not admit circuits of types (ii) and (iii$'$) if and only if it is quasibalanced.
\end{proposition}

	 \subsection{Matroid connection}\label{connect}

\begin{definition}\label{12} 
{\rm
A signed graph $G_{\sigma}$ is called \emph{frame connected} (\emph{$M$-connected} in \cite{k}) if $M(G_{\sigma})$ is connected or if $G_\sigma = K_1$.
It is called \emph{lift connected} if $L(G_{\sigma})$ is connected or if $G_\sigma = K_1$.
}
\end{definition}

\begin{lemma}\label{13}
A signed graph $G_{\sigma}$  is frame connected (lift connected) if and only if each pair of distinct edges $e$ and $\acute{e}$ from $G_{\sigma}$ are contained in a frame circuit (resp., lift circuit).
\end{lemma}
\begin{proof}
The proof is a direct consequence of Definitions \ref{3} and \ref{12} and the fact that a matroid is connected if and only if every two elements are together in a circuit \cite{o}.
\end{proof}

\begin{definition}\label{15} 
{\rm
The connected components of $M(G_{\sigma})$ induce subgraphs of $G_{\sigma}$, which are called \emph{frame components} of $G_{\sigma}$ (\emph{$M$-components} in \cite{k}). Thus, a frame component of $G_{\sigma}$ is a maximal subgraph without isolated vertices whose frame matroid is connected, or it is an isolated vertex.

The connected components of $L(G_{\sigma})$ induce subgraphs of $G_{\sigma}$, which are called \emph{lift components} of $G_{\sigma}$. Thus, a lift component of $G_{\sigma}$ is a maximal subgraph without isolated vertices whose lift matroid is connected, or it is an isolated vertex.
}
\end{definition}

\begin{remark}\label{frame-comp}
{\rm 
It is clear from Lemma \ref{13} that a frame component of $G_\sigma$ is a connected graph, but a lift component need not be a connected graph.  For instance, a lift circuit of type (iii$'$) is a disconnected graph but is lift connected.
}
\end{remark}

\begin{lemma}[\cite{z2}]\label{f-comp}
The frame components of $G_\sigma$ are the outer blocks, the constituents of any core that is an unbalanced necklace of balanced blocks, and the cores that are not such necklaces.

The lift components of $G_\sigma$ are the balanced blocks and the union of all unbalanced blocks, except that if there is only one unbalanced block and it is an unbalanced necklace of balanced blocks, then the constituents of the necklace are lift components.
\end{lemma}

\begin{definition}\label{b()}
{\rm
We indicate by $c(G_\sigma)$ the number of connected components of $G_\sigma$ and by $b(G_{\sigma})$ the number of balanced connected components  of $G_{\sigma}$.
}
\end{definition}

\begin{lemma}[\cite{z,z2}]\label{rank}
The rank in $M(G_\sigma)$ of a set $F \subseteq E$ is given by $r(F) = |V|-b(G_\sigma - (E-F))$.

The rank in $L(G_\sigma)$ is given by $r(F) = |V|-c(G_\sigma - (E-F))+\delta$ where $\delta=0$ if $G_\sigma$ is balanced and $1$ if $G_\sigma$ is unbalanced.
\end{lemma}

	 \subsection{Matroid isthmi}\label{m-isthmus}

\begin{definition}\label{16}
{\rm
An element $e$ of a matroid $M$ is a \emph{coloop} of $M$ if $\{e\}$ is a component of $M$ and has rank 1.
An edge $e$ of $G_{\sigma}$ that is a coloop of $M(G_{\sigma})$ will be called a \emph{frame isthmus of $G_{\sigma}$} (\emph{$M$-isthmus} in \cite{k}).
An edge $e$ of $G_{\sigma}$ that is a coloop of $L(G_{\sigma})$ will be called a \emph{lift isthmus of $G_{\sigma}$}.
}
\end{definition}

\begin{remark}\label{bal-isthmus}
{\rm 
If $G_{\sigma}$ is balanced, then the concepts of isthmus, frame isthmus, and lift isthmus coincide.
}
\end{remark}

\begin{remark}\label{19}
{\rm 
By a general property of matroids \cite{o}, an edge of $G_\sigma$ is a frame isthmus (resp., lift isthmus) if and only if it does not belong to any frame circuit (resp., lift circuit).
}
\end{remark}

\begin{lemma}\label{17}
An edge $e$ of $G_{\sigma}$ is a frame isthmus if and only if $e$ is an isthmus of a balanced connected component of $G_\sigma$, or an isthmus of an unbalanced connected component such that one of the two resulting connected components of $G_\sigma-e$ is balanced, or a balancing edge of an unbalanced connected component of $G_\sigma$.

An edge $e$ is a lift isthmus if and only if $e$ is an isthmus or is a balancing edge of an unbalanced connected component of $G_\sigma$.
\end{lemma}
\begin{proof}
An element $e$ of a matroid $M$ is a coloop if and only if the rank function satisfies $r(M-e)<r(M)$ (see \cite{o}).  

For $M(G_\sigma)$, $e$ is a coloop if and only if $b(G_{\sigma} - e) > b(G_{\sigma})$.  That means deleting $e$ either separates a balanced connected component into two connected components (so $e$ is an isthmus), or separates an unbalanced connected component into two connected components, one of which is balanced (so $e$ is an isthmus), or does not separate a connected component into two but does change a balanced connected component into an unbalanced connected component.

For $L(G_\sigma)$, $e$ is a coloop if and only if $c(G_{\sigma} - e)-c(G_\sigma) > \delta(G_{\sigma} - e)-\delta(G_\sigma)$.  If $e$ is an isthmus, the left-hand side is $+1$ and the right-hand side is at most $0$.  If $e$ is not an isthmus, the left-hand side is $0$ so $e$ is a lift isthmus if and only if $\delta(G_{\sigma} - e)=0$ and $\delta(G_\sigma)=1$.
\end{proof}

\begin{proposition}\label{21}
Let  $G_{\sigma}$ be a signed graph that is connected and quasibalanced. If  $e$ is a negative loop of $G_{\sigma}$, then $e$ is a frame isthmus and a lift isthmus.
\end{proposition}
\begin{proof}
If $e$ is not a frame or lift isthmus, it must belong to a frame or lift circuit of type (ii) or (iii) or (iii$'$).  This circuit contains a negative cycle $C$ that is not $\{e\}$.  Since $G_\sigma$ is quasibalanced, $C$ cannot exist.  Therefore, $e$ is a frame and lift isthmus.
\end{proof}

	\subsection{Sign connection and matroid connection}\label{s-m-conn}

Main results are the comparison of sign connection, sign components, and sign isthmi with matroid connection, components, and isthmi.

\begin{theorem}\label{compare}
Let $G_\sigma$ be a signed graph.  If $G_\sigma$ is frame connected, then it is sign connected.  If it is sign connected, it may not be frame connected.

If $G_\sigma$ is lift connected, then each connected component is sign connected and is a sign component of $G_\sigma$.  If it is sign connected, it may not be lift connected.
\end{theorem}
\begin{proof}
Compare the characterization of sign components in Theorem \ref{31} to those of frame and lift components in Lemma \ref{f-comp}.  In particular, an outer block of $G_\sigma$ is sign disconnected and an isthmus in a sign-connected component of $G_\sigma$ is a lift isthmus of $G_\sigma$ and is not lift-connected to the rest of $G_\sigma$.
\end{proof}

\begin{proposition}\label{35}
Let $G_{\sigma}$ be a sign-connected signed graph. If for some edge $e$ of $G_{\sigma}$ each connected component of $G_{\sigma} - e$ is sign connected and has more than one vertex or supports a negative loop, then $G_{\sigma}$ is not quasibalanced.
\end{proposition}
\begin{proof}
Since $G_\sigma$ is sign connected, it is connected.  
If for some edge $e$ of $G_{\sigma}$ the connected components of $G_{\sigma} - e$ are sign connected, then each one is unbalanced since we assumed it has more than one vertex or supports a negative loop.  
This implies that each connected component admits a negative cycle, thus $G_{\sigma}$ is not quasibalanced.
\end{proof}

\begin{proposition}\label{33}
In a sign-connected signed graph, every frame or lift isthmus $e$ is a sign isthmus, except that when $|V|=1$, $e$ is a negative loop, and $G_\sigma$ contains no other negative loop, then $e$ is not a sign isthmus.
\end{proposition}
\begin{proof}
By Theorem \ref{31}, $G_\sigma$ is connected.  Let $e$ be a frame isthmus in $G_{\sigma}$.  
If $e$ is an isthmus of $G_{\sigma}$, it is a sign isthmus of $G_{\sigma}$.
If $e$ is not an isthmus of $G_{\sigma}$, then it is a balancing edge of $G_\sigma$.  If $|V|>1$, this implies that $G_\sigma$ is not sign connected, from which the edge $e$ is a sign isthmus.  If $|V|=1$, $e$ is a negative loop and  $G_{\sigma}-e$ is sign connected, so $e$ is not a sign isthmus.

Let $e$ be a lift isthmus.  If it is either an isthmus or a balancing edge of $G_{\sigma}$, it is a sign isthmus if $|V|>1$ but not if $|V|=1$ for the same reason as above.
\end{proof}

\begin{remark}\label{sign-frame-isthmus}
{\rm   
A sign isthmus that is not a negative loop is not always a frame isthmus.  Let $G_\sigma$ be connected with an isthmus $e$ such that $G_\sigma-e$ has two unbalanced components.  Then $e$ is a sign isthmus but not a frame isthmus.
However, by Proposition \ref{36} and Lemma \ref{17} every sign isthmus is a lift isthmus.
}
\end{remark}

\begin{theorem}\label{34}
Let $G_{\sigma}$ be a sign-connected signed graph and $e$ a sign isthmus that is an isthmus. The connected components of $G_{\sigma} - e$ are sign connected if and only if the edge $e$ is not a frame isthmus.
\end{theorem}
\begin{proof}
A connected component of $G_{\sigma} - e$ that is sign connected is unbalanced or it is balanced and has only one vertex.  If both connected components of $G_{\sigma} - e$ are unbalanced, then $e$ is in a frame circuit of Type (iii) so it is not a frame isthmus.  If one of them is balanced with vertex $v$, then $v$ is not sign connected to the rest of $G_\sigma$, contrary to hypothesis, so this is impossible.
\end{proof}

\begin{remark}\label{sign isth}
{\rm 
Under the hypotheses of Theorem \ref{34}, the connected components of $G_{\sigma} - e$ are sign connected if and only if they are both unbalanced.  This is demonstrated by the same proof.
}
\end{remark}

	\section{Signed graphs without positive cycles}\label{nopos}

In this section we specialize to signed graphs that have no positive cycles.  

\begin{definition}[{\cite[Example 6.2]{z1}}]\label{contra}
{\rm 
A signed graph without positive cycles is called \emph{contrabalanced}.  
}
\end{definition}

\begin{definition}\label{cactusforest}
{\rm 
A \emph{cactus forest} is a graph in which every block is a cycle, an isthmus, or $K_1$.  

A \emph{theta graph} is a graph that consists of three elementary chains that are disjoint except that they all have the same two end vertices.
}
\end{definition}

\begin{theorem}\label{cacti}
A signed graph has no positive cycles if and only if it is a contrabalanced cactus forest.
\end{theorem}

(We saw this theorem stated independently in a recently published or submitted paper, but we do not recall which paper.)

\begin{proof}
A signed graph $G_\sigma$ contains a positive cycle if contains a theta graph, because every theta graph in a signed graph contains an odd number of positive cycles \cite[Theorem 6]{zc}.  If $G_\sigma$ is contrabalanced, therefore every block is an isthmus, a cycle, or $K_1$.  
\end{proof}

\begin{theorem}\label{sign-contra}
Let $G_{\sigma}$ be a connected, contrabalanced signed graph with at least two vertices.  Then $G_{\sigma}$ is sign connected if and only if it has a cycle.

If $G_\sigma$ has exactly one cycle, then every edge is a sign isthmus.  If $G_\sigma$ has at least two cycles, an edge is a sign isthmus if and only if it is an isthmus.
\end{theorem}

\begin{proof}
\emph{Sign connection}:  The signed graph $G_\sigma$ is balanced if and only if it has no cycles.  By Theorem \ref{31} it is sign connected if and only if it is unbalanced.

\emph{Sign isthmi}:  If $G_\sigma$ has any cycles, then it is unbalanced, so by Proposition \ref{36} it is sign connected.  Let $e$ be an edge of $G_\sigma$.  
If $G_\sigma$ has at least two cycles, $G_\sigma-e$ is unbalanced; therefore, it is sign connected if and only if it is connected.  Thus, $e$ is a sign isthmus if and only if it is an isthmus.  
If $G_\sigma$ has exactly one cycle, then it is unbalanced but $G_\sigma-e$ is either balanced, if $e$ is in the cycle, or disconnected, if $e$ is not in the cycle.  Thus, every edge is a sign isthmus.
\end{proof}

For contrast we provide the similar results for frame and lift connection.  Note that 

\begin{theorem}\label{25}
Let $G_{\sigma}$ be a connected, contrabalanced signed graph with at least two edges.  The following properties are equivalent:
\begin{enumerate}[{\rm (1)}]
\item $G_{\sigma}$ is frame connected.
\item $G_{\sigma}$ contains no frame isthmus.
\item $G_{\sigma}$ has at least two cycles and no pendant edge.
\end{enumerate}
If $G_\sigma$ does not have these properties, then every edge is a frame isthmus.
\end{theorem}

\begin{proof}
(1) $\Rightarrow$ (2):  
Since $G_{\sigma}$ is frame connected and has at least two edges, it cannot have a frame isthmus.

(2) $\Rightarrow$ (3):  
Since $G_\sigma$ has no frame isthmus, every edge is in a frame circuit of type (ii) or (iii).  It follows that $G_\sigma$ has at least two cycles and no pendant edges.

(3) $\Rightarrow$ (1):
Every block of $G_\sigma$ is either a negative cycle or an isthmus on a path between two negative cycles.  Therefore, $M(G_\sigma)$ is connected, by Lemma \ref{f-comp}.

\emph{Frame isthmi}:  If (3) is not true, then there are no frame circuits so every edge is a frame isthmus, by Lemma \ref{f-comp}.
\end{proof}

\begin{theorem}\label{liftcactus}
Let $G_{\sigma}$ be a contrabalanced signed graph with at least two edges and without isolated vertex.  The following properties are equivalent:
\begin{enumerate}[{\rm (1)}]
\item $G_{\sigma}$ is lift connected.
\item $G_{\sigma}$ contains no lift isthmus.
\item $G_{\sigma}$ has at least two cycles and no isthmus.
\end{enumerate}
If $G_\sigma$ does not have at least two cycles, then every edge is a lift isthmus.
\end{theorem}

\begin{proof}
(1) $\Rightarrow$ (2):  
Since $G_{\sigma}$ is lift connected and has at least two edges, it cannot have a lift isthmus.

(2) $\Rightarrow$ (3):  
Since $G_\sigma$ has no lift isthmus, every edge is in a lift circuit of type (ii) or (iii$'$).  It follows that $G_\sigma$ has at least two cycles and no isthmus.

(3) $\Rightarrow$ (1):
Every pair of cycles forms a lift circuit of type (ii) or (iii$'$).  Therefore, every pair of edges is in a lift circuit, so $L(G_\sigma)$ is connected, by Lemma \ref{13}.

\emph{Lift isthmi}:  Every isthmus is a lift isthmus by Lemma \ref{17}.  If $G_\sigma$ has only one cycle, every edge of that cycle is a balancing edge and therefore is a lift isthmus by Lemma \ref{17}.
\end{proof}

	\section{Signed graphs where every edge is negative}\label{anti}

Let $G_\sigma$ be a signed graph such that every edge has sign $-1$.  Then a chain is positive if it has even length and negative if it has odd length, so two vertices are sign connected if and only if they are joined by both even and odd chains.  The essence of this signature is the property of antibalance s(uggested by Harary \cite{ha}).

\begin{definition}\label{antibal}
{\rm 
A signed graph $G_\sigma$ is called \emph{antibalanced} if $G_{-\sigma}$ is balanced.  Equivalently, 
$G_\sigma$ switches to have all negative signs.  Also equivalently, all positive cycles have even length and all negative cycles have odd length.
}
\end{definition}

\begin{remark}\label{antibal-bal}
{\rm
An antibalanced graph $G_\sigma$ is balanced if and only if $G$ is bipartite.  Thus, antibalanced signed graphs are a signed generalization of bipartite graphs.
}
\end{remark}

Henceforth in this section we assume that all edges are negative.  A signed graph in which every edge is negative is denoted by $G_{-}$ and is called \emph{all negative}.

\begin{definition}\label{parityconn}
{\rm 
In an unsigned graph $G$, two vertices $v$ and $w$ are called \emph{parity connected} if there exist both an odd-length and an even-length chain between them.  Equivalently, $v$ and $w$ are sign connected in $G_{-}$.  The graph $G$ is called \emph{parity connected} if every two vertices are parity connected in $G$; equivalently, $G_{-}$ is sign connected.
}
\end{definition}

All our previous results apply to unsigned graphs $G$ by specializing to $G_{-}$, replacing ``positive'', ``negative'', and ``balanced'' by ``even'', ``odd'', and ``bipartite'', respectively.  Sign connection becomes parity connection.  Signed graphs without positive cycles become graphs with no even cycles.  The following results are examples of this substitution.

\begin{proposition}\label{40-}
A hypercyclic chain between $x$ and $y$ in an all-negative signed graph $G_{-}$ is any chain in $G$ between $x$ and $y$ that has the form in Figure \ref{FIG3} with a cycle of odd length.
\end{proposition}
\begin{proof}
A hypercyclic $\varepsilon$-chain contains a negative cycle, which in $G_\sigma$ means an odd cycle.  The rest is a restatement of Definition \ref{6} and Proposition \ref{hypercyclic}.
\end{proof}

\begin{theorem}\label{41-}
A connected unsigned graph with at least two vertices is parity connected if and only if it is not bipartite.
\end{theorem}
\begin{proof}
Let $G$ be the connected graph.  It is parity connected if and only if $G_{-}$ is sign connected and $G$ is not bipartite if and only if $G_{-}$ is unbalanced.  By Theorem \ref{31}, $G_{-}$ is sign connected if and only if it is unbalanced.
\end{proof}

	\section{Extensions}\label{open}

	\subsection{Signed path connection}\label{signpath}

An unsigned graph is connected if every pair of vertices are joined by a chain, or equivalently by a path (an elementary chain).  In sign connection chains cannot be replaced by paths.  Define signed path connection as the existence of paths of both signs between a pair of vertices.  It is is a nontransitive relation on the set of vertices; therefore it is necessary to take the transitive closure to get an equivalence relation.  Since the transitive closure of path connection is connection by chains, it follows that signed path connection is not intrinsically interesting.

	\subsection{Positive and negative connection}\label{posneg}

Instead of asking for chains of both signs, we might ask only for chains of one fixed sign.

\begin{definition}\label{posneg-conn}
{\rm 
A signed graph $G_\sigma$ is \emph{positively connected} if every pair of distinct vertices is joined by a chain of positive sign.  The \emph{positively connected components} of $G_\sigma$ are the maximal positively connected subgraphs of $G_\sigma$.

Vertices $u,v$ in $G_\sigma$ are \emph{negatively connected} if they are joined by a negative chain.  The relation of being negatively connected is not transitive, in general.  Therefore, we define a \emph{negatively connected component} of $G_\sigma$ to be a maximal subgraph in which every pair of vertices, $u$ and $v$, is negatively connected, or has a third vertex $w$ such that $u$ and $v$ are negatively connected to $w$.
}
\end{definition}

\begin{remark}\label{posneg-sw}
{\rm 
These notions of connection are not invariant under switching.  Therefore, they are intrinsically different from sign connection.  However, they are closely related to it.
}
\end{remark}

\begin{proposition}\label{posneg-comp}
The positively connected components of $G_\sigma$ are the  unbalanced connected components of $G_\sigma$, the subgraphs generated by the sets of the Harary bipartition of each balanced connected component in which not all edges are positive, and the connected components in which all edges are positive.

The negatively connected components of $G_\sigma$ are the unbalanced connected components of $G_\sigma$, the balanced connected components that are not all positive, and the separate vertices of the connected components in which all edges are positive.
\end{proposition}

\begin{proof}
A sign-connected component of $G_\sigma$ is both positively connected and negatively connected.  Therefore (by Theorem \ref{31}), every unbalanced connected component is a positively connected component and a negatively connected component.  

A balanced connected component that is all positive is positively connected, and none of its vertices are negatively connected.  Thus, the component is a positively connected component and each vertex is a negatively connected component.

If a balanced connected component $G_\sigma'$ is not all positive, let $W,\ V-W$ be the sets of its Harary bipartition (Definition \ref{harbi}).  The negative edges are the edges connecting $W$ to $V-W$.  A positive chain that begins in $W$ must end in $W$, so the vertices of $W$ are not positively connected to those of $V-W$, but since there is a path joining any two vertices of the component, all vertices in $W$ are positively connected to each other.  Therefore, $W$ generates a positively connected component of $G_\sigma$.  Similarly, $V-W$ also generates a positively connected component.  

Since in $G_\sigma'$ a negative chain connects each vertex in $W$ to each vertex in $V-W$, $G_\sigma'$ is a negatively connected component.
\end{proof}

\begin{remark}\label{posneg-paths}
{\rm 
If the chains are required to be paths, the definition of positive (resp., negative) connection is nontransitive; therefore it is necessary to take the transitive closure.  There is no difference between this transitive closure and positive (resp., negative) connection defined by chains.
}
\end{remark}


\end{document}